\RequirePackage{amsmath}
\documentclass[smallcondensed,envcountsame,envcountsect]{svjour3}

\usepackage{mathptmx}
\usepackage{amssymb}

\newcommand{\F}{\mathbb{F}}
\newcommand{\Z}{\mathbb{Z}}

\smartqed
\title{A Swan-like note for a family of binary pentanomials}
\journalname{AAECC}
\author{Giorgos Kapetanakis}
\institute{
	Giorgos Kapetanakis
	\at Faculty of Engineering and Natural Sciences, Sabanc\i{} \"{U}niversitesi. Ortha Mahalle, Tuzla 34956, \.{I}stanbul, Turkey \\\email{gnkapet@gmail.com}}

\begin{document}
\maketitle

\begin{abstract}
In this note, we employ the techniques of Swan ({\em Pacific J. Math.} {12}(3): 1099--1106, 1962) with the purpose of studying the parity of the number of the irreducible factors of the penatomial $X^n+X^{3s}+X^{2s}+X^{s}+1\in\F_2[X]$, where $s$ is even and $n>3s$. Our results imply that if $n \not\equiv \pm 1 \pmod{8}$, then the polynomial in question is reducible.
\keywords{Swan-like \and binary field \and pentanomial}
\subclass{11T06 \and 11C08}
\end{abstract}
\section{Introduction}
Let $\F_{2^n}$ be the extension of degree $n$ of $\F_2$, the binary field. This field has numerous applications in practical areas, like cryptography or coding theory. The most common way to represent such fields is to utilize a polynomial basis, in which case an irreducible polynomial of degree $n$ over $\F_2$ is required. In particular, there are obvious computational advantages in choosing low-weight polynomials, that is polynomials with as few non-zero coefficients as possible. Namely, it is advised, see \cite{ieee00}, to favor trinomials or pentanomials. It is then natural to study the irreducibility (or equivalently the lack of it) of binary trinomials and pentanomials. We refer the interested reader to \cite[Section~3.3]{lidlniederreiter97} and \cite[Chapter~3]{mullenpanario13} and the references therein.

The interest for studying pentanomials was renewed as a result of the computational advantages that certain families of pentanomials bear \cite{banegascustodiopanario18,reyhanimasoleshhasan04,rodriguezhenriquezkoc03}. More specifically, the usage of the pentanomial
\[ X^n + X^{n-s} + X^{n-2s} + X^{n-3s} + 1 \in\F_2[X] , \]
where $s < n/3$, has been proposed. In particular, several authors \cite{reyhanimasoleshhasan04,zhangparhi01} have proved the computational advantages that the usage of such polynomials (known as \emph{class 2} pentanomials), as the corresponding Mastrovito multipliers have low complexity. On the other hand, the number of irreducible polynomials within the family of class 2 pentanomials has been observed to be abnormally small, see \cite[Remark~4]{reyhanimasoleshhasan04}. In this work, we study class 2 pentanomials, when $s$ is even and we prove the following.
\begin{theorem}\label{thm:our1}
Let $f(X)=X^n+X^{n-s}+X^{n-2s}+X^{n-3s}+1\in\F_2[X]$ and $g(X) = X^n+X^{3s}+X^{2s}+X^s+1\in\F_2[X]$, with $s$ even and $n>3s$. If $n \not\equiv \pm 1 \pmod{8}$, then $f$ and $g$ are reducible.
\end{theorem}
In particular, our results explain the small number of irreducible polynomials within the family of class 2 pentanomials, as a large numbers of representatives of this family are \emph{a priori} reducible.

Our method is based on Swan's \cite{swan62} techniques, who studied the parity of the number of irreducible factors of binary trinomials and proved the theorem below.
\begin{theorem}[Swan]\label{thm:swan1}
Let $n > k > 0$. Assume exactly one of $n$, $k$ is odd. Then $X^n + X^k + 1\in\F_2[X]$ has an even number of factors (and hence is reducible) in the following cases.
\begin{enumerate}
\item $n$ is even, $k$ is odd, $n\neq 2k$ and $nk/ 2\equiv 0$ or $1\pmod{4}$.
\item $n$ is odd, $k$ is even, $k\nmid 2n$ and $n\equiv \pm 3 \pmod 8$.
\item $n$ is odd, $k$ is even, $k\mid 2n$ and $n\equiv \pm 1 \pmod 8$.
\end{enumerate}
In all other cases $X^n + X^k + 1$ has an odd number of factors over $\F_2$.
\end{theorem}
The above has been extended by several authors \cite{bluher06,fredricksenhalessweet86,halesnewhart06,hansonpanariothomson11}. In addition, several results are known for trinomials \cite{hansonpanariothomson11}, tetranomials \cite{halesnewhart06} and certain families of binary pentanomials \cite{ahmandi06,koepfkim09}. 

We conclude this note with some observations about the distribution of the binary irreducible polynomials of the form $X^n+X^{n-s}+X^{n-2s}+X^{n-3s}+1$, when $s$ is even and $n$ is small.
\section{Preliminaries}
The \emph{discriminant} of the monic polynomial $F$ over an integral domain is defined as
\[
D(F) := \prod_{1\leq i<j\leq n} (\alpha_i-\alpha_j)^2,
\]
where $\alpha_1,\ldots ,\alpha_n$ are the roots of $F$ counted with multiplicity and $\deg(F)=n$. By using standard properties of the discriminant, see \cite{halesnewhart06}, one can show the following alternative formula for $D(F)$.
\begin{lemma}\label{lemma:discr}
Let $F$ be as above, with derivative $F'$ and the additional assumption that its constant term is equal to $1$. Further, let $H(X) := nF(X) -XF'(X)$. Then
\[
D(F)
 = (-1)^{\frac{n(n-1)}{2}} \prod_{i=1}^n H(\alpha_i) .
\]
\end{lemma}
\begin{proof}
Since $F(X) = \prod_{i=1}^n (X-\alpha_i)$, hence 
\[ F'(X) = \sum_{i=1}^n \left( \prod_{j\neq i} (X-\alpha_j) \right), \]
that is $F'(\alpha_i) = \prod_{j\neq i} (\alpha_i - \alpha_j)$, for every $1\leq i\leq n$. Then, we get that
\[
\prod_{i=1}^n F'(\alpha_i) = \prod_{i=1}^n \prod_{\substack{1\leq j\leq n \\ j\neq i}} (\alpha_i-\alpha_j) = (-1)^{\frac{n(n-1)}{2}} \prod_{1\leq i<j\leq n} (\alpha_i-\alpha_j)^2 = (-1)^{\frac{n(n-1)}{2}} D(F).
\]
The result follows after observing that
\[
\prod_{i=1}^n H(\alpha_i) = (-1)^n \alpha_1\cdots \alpha_n \prod_{i=1}^n F'(\alpha_i)
\]
and that $(-1)^n\alpha_1\cdots\alpha_n=1$.
\qed\end{proof}
Towards the proof of \ref{thm:swan1}, Swan's main tool was the following.
\begin{lemma}[\cite{swan62}, Corollary~3]\label{lemma:swan2}
Let $f\in\F_2[X]$ such that $f$ is square-free and let $t_f$ denote the number of irreducible factors of $f$ over $\F_2$. If $F$ is a lift of $f$ in $\Z$, then $t_f\equiv \deg(f) \pmod{2}$ if and only of $D(F)\equiv 1 \pmod{8}$.
\end{lemma}
Motivated by the above, we calculate $D(F)$ modulo $8$. Following the work in \cite{halesnewhart06}, we begin with proving the lemma below.
\begin{lemma}\label{lemma:ps_id}
Let $F\in\Z[X]$ be a monic polynomial such that $F(0)=1$ with $\alpha_1,\ldots ,\alpha_n$ its roots counted with multiplicity. Then for every $X$, with absolute value small enough,
\[ X F'(X) \sum_{i=0}^\infty (-1)^{i+1} (F(X)-1)^i = \sum_{i=1}^\infty S_{-i}^{(F)} X^i , \]
where for every $j\in\Z$, $S_j^{(F)} := \sum_{i=1}^n \alpha_i^j$.
\end{lemma}
\begin{proof}
Notice that for every $j\in\{1,\ldots ,n\}$ we have that $\alpha_j\neq 0$. Also, notice that
\[ - \frac{XF'(X)}{F(X)} = X \cdot \sum_{j=1}^k \frac{1}{\alpha_j} \cdot \frac{1}{1-X/\alpha_j} . \]
Further, notice that for every $j$ and $X$ with small enough absolute value,
\[ \frac{1}{1-X/\alpha_j} = \sum_{i=0}^\infty \left( \frac{X}{\alpha_j} \right)^i \]
and that for $X$ with small enough absolute value,
\[ \frac{1}{F(X)} = \frac{1}{1 - (1 - F(X))} = \sum_{i=0}^\infty (1-F(X))^i = \sum_{i=0}^\infty (-1)^i (F(X)-1)^i . \]
It follows that for $X$ with small enough absolute value,
\[ XF'(X) \sum_{i=0}^\infty (-1)^{i+1} (F(X)-1)^i = \sum_{i=1}^\infty \left( \sum_{j=1}^n \frac{1}{\alpha_j^i} \right) X^i \]
and the result follows.
\qed\end{proof}

Another useful tool for computing the discriminant of polynomials is the following, see \cite[Theorem~1.75]{lidlniederreiter97}.
\begin{theorem}[Newton's formula]\label{thm:newton}
Let $F(X)=X^n + F_{n-1}X^{n-1} + \cdots + F_1X + F_0$ be a polynomial over some field with roots $\alpha_1,\ldots ,\alpha_n$, counted with multiplicity and fix some $m\in\Z$. Further, for any integer $t$, define $S_t^{(F)} = \sum_{j=1}^n \alpha_j^t$, then
\[
S_m^{(F)} + F_{n-1}S_{m-1}^{(F)} + \cdots + F_{n-l+1}S_{m-l+1}^{(F)} + \frac lm F_{n-l}S_{m-l}^{(F)} = 0 ,
\]
where $l := \min (m,n)$.
\end{theorem}
\section{Proof of the main theorem}
From now on let $s$ be an even positive integer and $n>3s$. Also, set $f(X) := X^n + X^{n-s} + X^{n-2s} + X^{n-3s} + 1\in\F_2[X]$, the typical class 2 pentanomial, and $F(X) := X^n + X^{n-s} + X^{n-2s} + X^{n-3s} + 1\in\Z[X]$. It is clear that $F$ is a lift of $f$ in $\Z[X]$. The case when $n$ is even is trivial, since then clearly $f$ is a square in $\F_2$, hence reducible. So from now on we will additionally assume that $n$ is odd.

Also, if $\alpha_1,\ldots ,\alpha_n$ are the roots of $F$ counted with multiplicity, then for every integer $m$, set
\[ S_m := S_m^{(F)} = \sum_{j=1}^n \alpha_j^m \quad\text{and}\quad T_m := \sum_{1\leq i<j\leq n} (\alpha_i\alpha_j)^m , \]
while one easily verifies that the above, as well as similar expressions, are symmetric expressions of the roots the $F$. Furthermore, it is well-known that such expressions can be written as polynomials of the elementary symmetric functions with integer coefficients. This means that, since $F$ is monic and has integer coefficients, by Vieta's formulas, the elementary symmetric functions have integer values, which implies that all symmetric expressions of the roots the $F$ with integer coefficients have integer values.
Additionally, set
\[ H(X) := nF(X) -XF'(X) = sX^{n-s}+2sX^{n-2s}+3sX^{n-3s}+n . \]

It is clear from Lemma~\ref{lemma:swan2} that we are interested in computing $D(F)$ modulo $8$.
Towards this end, we compute
\begin{multline*}
\prod_{i=1}^n H(\alpha_i) = n^n + n^{n-1}sS_{n-s} + 2n^{n-1}sS_{n-2s} + 3n^{n-1}sS_{n-3s} + \\
n^{n-2}s^2T_{n-s} + n^{n-2}s^2T_{n-3s} + 3n^{n-2}s^2(S_{n-s}S_{n-3s} - S_{2n-4s}) + 8K ,
\end{multline*}
for some $K\in\Z$, where we note that $s$ is even, hence the terms that include $4s$, $2s^2$ or $s^k$ for $k\geq 3$ are divisible by $8$, so one can sum them up as $8K$. Since $n$ is odd, we have that $n^2 \equiv 1 \pmod{8}$, hence
\begin{multline}\label{eq:prod_h1}
\prod_{i=1}^n H(\alpha_i) \equiv n + sS_{n-s} + 2sS_{n-2s} + 3sS_{n-3s} + ns^2T_{n-s} + ns^2T_{n-3s} \\ + 3ns^2(S_{n-s}S_{n-3s} - S_{2n-4s}) \pmod{8}.
\end{multline}

Next, we apply Lemma~\ref{lemma:ps_id} on the reciprocal of $F$ and, for small enough $X$, we get:
\[
(nX^n+sX^s+2sX^{2s}+3sX^{3s}) \sum_{i=0}^\infty (-1)^{i+1} (X^n+X^s+X^{2s}+X^{3s})^i = \sum_{i=1}^\infty S_iX^i .
\]
In the LHS of the above equation, we observe that the only non-zero coefficients of terms with degree smaller than $n$ have in fact a degree that is a multiple of $s$, i.e. are even since $s$ is even. This implies that all the terms of odd degree that are smaller than $n$ are zero.
In particular, $n-ks$ is odd and strictly smaller than $n$, hence in the LHS of the equation, the coefficient of $X^{n-ks}$ is zero. It follows that the same holds for the RHS, that is
\begin{equation}\label{eq:s_n-ks=0}
S_{n-ks}=0.
\end{equation}
Now, Eq.~\eqref{eq:prod_h1} yields
\begin{equation}\label{eq:prod_h2}
\prod_{i=1}^n H(\alpha_i) \equiv n + ns^2(T_{n-s} +T_{n-3s}) -3ns^2S_{2n-4s} \pmod{8}.
\end{equation}
Furthermore, since we are only interested in the value of $D(F)$ modulo $8$ and since $s$ is even, Eq.~\eqref{eq:prod_h2} implies that for our purposes it suffices to compute $T_{n-s} + T_{n-3s}$ and $S_{2n-4s}$ modulo $2$. First, we observe that
\begin{equation}\label{eq:s_2n-4s}
S_{2n-4s} \equiv \sum_{i=1}^n \alpha_i^{2n-4s} \equiv \left( \sum_{i=1}^n \alpha_i^{n-2s} \right)^2 \equiv (S_{n-2s})^2 \equiv 0 \pmod{2},
\end{equation}
from Eq.~\eqref{eq:s_n-ks=0}.
Also, by applying Theorem~\ref{thm:newton} for the polynomial in question for $m=2n-2s$ and $m=2n-3s$, we get
\[ \begin{cases}
S_{2n-2s} + S_{2n-3s} + S_{2n-4s} + S_{2n-5s} + S_{n-2s} = 0, & \text{and} \\
S_{2n-3s} + S_{2n-4s} + S_{2n-5s} + S_{2n-6s} + S_{n-3s} = 0 .
\end{cases} \]
By subtracting the above equations, and with Eq.~\eqref{eq:s_n-ks=0} in mind, we conclude that $S_{2n-2s} = S_{2n-6s}$. This combined with the identity $T_k = (S_k^2 - S_{2k})/2$ implies $T_{n-s}=T_{n-3s}$, hence $T_{n-s}+T_{n-3s} \equiv 0 \pmod{2}$. The latter, along with Eqs.~\eqref{eq:prod_h2} and \eqref{eq:s_2n-4s} yields
\[
\prod_{i=1}^n H(\alpha_i) \equiv n \pmod{8} .
\]
This, combined with Lemma.~\ref{lemma:discr} gives
\[
D(F) \equiv (-1)^{n(n-1)/2} n \equiv \begin{cases}
1 \pmod{8} , & \text{if } n \equiv \pm 1\pmod{8} , \\ 5 \pmod{8} , & \text{if } n \equiv \pm 3 \pmod{8} .
\end{cases}
\]
The combination of the above with Lemma~\ref{lemma:swan2} implies the following.
\begin{proposition}\label{propo:our1}
Set $f(X) = X^n + X^{n-s} + X^{n-2s} + X^{n-3s} \in \F_2[X]$ and let $t_f$ be the number of irreducible factors of $f$ in $\F_2[X]$. Then 
\[
t_f \equiv \begin{cases}
1 \pmod{2} , & \text{if } n\equiv \pm 1 \pmod{8} , \\ 0 \pmod{2} , & \text{otherwise}.
\end{cases}
\]
\end{proposition}
Theorem~\ref{thm:our1} is an immediate consequence of Proposition~\ref{propo:our1}, once we notice that $X^n+X^{3s}+X^{2s}+X^s+1\in\F_2[X]$ and $X^n+X^{n-s}+X^{n-2s}+X^{n-3s}+1\in\F_2[X]$ are reciprocal to each other, that is they share reducibility and irreducibility.
\section{Concluding remarks}
In Theorem~\ref{thm:our1}, we proved that all the irreducible polynomials of the form $X^n+X^{n-s}+X^{n-2s}+X^{n-3s}+1\in\F_2[X]$, with $s$ even, satisfy $n\equiv \pm 1\pmod{8}$. In order to obtain some insight about the distribution of irreducible polynomials of that form, a computer search was performed with the computer algebra system \textsc{SageMath}. Namely, we did an exhaustive search of all possible values for odd $n$ and even $s$, for $7\leq n< 3000$ and a total of 374,250 polynomials were checked for irreducibility. The search revealed that 804 of them are irreducible\footnote{As a confirmation of our results, we verified our results with \textsc{Magma} for $7\leq n\leq 1000$ and found identical results} and the pairs $(n,s)$ that yielded irreducible polynomials are presented in Table~\ref{table2}.

The results confirm that all 804 class 2 irreducible polynomials satisfy $n\equiv \pm 1\pmod{8}$. Out of these, roughly half, i.e. 401 out of 804, satisfy $n\equiv 1\pmod{8}$ and the other 403 satisfy $n\equiv -1\pmod{8}$.
Another interesting observation is that the irreducible polynomials seem to be also uniformly distributed among the even values of $s$ modulo $8$. In particular, 214 polynomials satisfy $s\equiv 2\pmod{8}$, 188 satisfy $s\equiv 4\pmod{8}$, 198 satisfy $s\equiv 6\pmod{8}$ and the rest 204 satisfy $s\equiv 0\pmod{8}$.

Regarding the frequency of irreducibility within the class 2 pentanomials tested, we have observe that out of the 749 values of $n$ considered (i.e. such that $7\leq n\leq 3000$ and $n \equiv \pm 1 \pmod{8}$), 408 of them yield irreducible polynomials for some $s$ and the other 341 do not. What is more interesting however, is that it seems to be more possible for a class 2 pentanomial with $n\equiv \pm 1 \pmod{8}$ and $s$ even to be irreducible than an arbitrarily chosen binary polynomial of degree $n$. More precisely, out of the 187,125 class 2 pentanomials with those specifications we tested, 804 turned out to be irreducible, hence we had a frequency $\sim 0.43\%$. In contrast, the corresponding frequency for an arbitrary binary polynomial of the same degrees is $\sim 0.13\%$, even if we exclude the obviously reducible polynomials (i.e. those with roots in $\F_2$).

An extension of Theorem~\ref{thm:our1} for $s$ odd does not hold, as a quick computer search verifies. Namely, we performed an exhaustive search in the  interval $7\leq n< 3000$, for $s$ odd. We checked 749,996 polynomials for irreducibility and identified 1707 irreducible polynomials. We had a sample roughly double the size compared to the one from our previous test and we found roughly double the number of irreducible polynomials. This suggests that if one excludes the obviously reducible class 2 pentanomials (when $n$ and $s$ are even), then the possibility of the arbitrary class 2 pentanomial to be irreducible is almost the same for both $s$ odd and even. However, after also counting Theorem~\ref{thm:our1} in, we see that one such polynomial with $s$ even and $n\equiv \pm 3 \pmod{8}$ looks more likely to be irreducible than one with $s$ odd, while it is worth mentioning that irreducibility seems to be close to uniformly distributed for different values of (odd) $s$ modulo 8, but we found zero pairs $(n,s)$ with $8\mid n$ and very few with $n\equiv 3,5 \pmod{8}$.

We conclude this note with two final remarks. First, since reducibility implies the lack of primitivity, the pentanomials described in Theorem~\ref{thm:our1} are also non-primitive. A quick computer test suggests that among the irreducible pentanomials of Table~\ref{table2}, one finds a reasonable number of primitive polynomials without any obvious pattern. Second, we note that class 2 pentanomials share many similarities with \emph{equally spaced polynomials}, that is polynomials of the form $f(X^n)\in\F_2[X]$. Such polynomials can also be used to construct low complexity multipliers \cite{wuhasan98}. It is natural to wonder about connections between the two families or special properties of pentanomials that belong in both families. 
\begin{table}
\begin{center}\tiny
\begin{tabular}{ll|ll|ll|ll|ll}
 $n$ & $s$ & $n$ & $s$ & $n$ & $s$ & $n$ & $s$ & $n$ & $s$  \\ \hline
7 & 2 &
17 & 2, 4 &
23 & 6 &
25 & 6 &
31 & 2, 6, 8 \\
47 & 14 &
49 & 4 &
55 & 8, 16 &
65 & 6 &
71 & 2, 6, 12 \\
73 & 14, 16 &
79 & 20 &
95 & 26, 28 &
97 & 2, 4 &
103 & 10, 24, 30 \\
113 & 10 &
121 & 6, 10 &
127 & 10, 40, 42 &
137 & 34 &
151 & 22, 28, 36, 40 \\
161 & 6, 20 &
167 & 2, 30, 36, 44 &
169 & 14, 28 &
175 & 2, 6 &
185 & 8, 48 \\
191 & 6, 40 &
193 & 36, 40 &
199 & 44 &
209 & 2, 54 &
215 & 38, 46, 64 \\
217 & 22, 44 &
223 & 44 &
239 & 12 &
247 & 34 &
257 & 4, 16, 64, 72 \\
265 & 14, 46 &
287 & 54, 72 &
289 & 12, 28 &
295 & 16 &
305 & 34 \\
313 & 64, 78 &
319 & 12 &
329 & 18 &
337 & 66, 94 &
343 & 46 \\
353 & 46, 60, 70, 86 &
377 & 112 &
383 & 30, 36 &
385 & 2, 8, 18 &
391 & 120 \\
407 & 112 &
415 & 34, 84 &
425 & 4, 14, 22, 78 &
431 & 40 &
433 & 124 \\
439 & 52, 98, 102, 130 &
449 & 94 &
457 & 70, 80, 132 &
463 & 56 &
481 & 46 \\
487 & 120 &
497 & 26, 72, 76 &
505 & 52, 58 &
511 & 72, 160 &
521 & 16, 56 \\
527 & 66, 96, 160 &
529 & 14, 38, 124 &
551 & 80 &
553 & 86, 148 &
559 & 70 \\
569 & 70, 164 &
575 & 86 &
577 & 184 &
593 & 36, 158 &
599 & 10, 70 \\
623 & 104, 124, 146 &
625 & 52, 164 &
631 & 108 &
641 & 12, 118, 182, 210 &
647 & 50, 104, 144, 214 \\
649 & 192, 204 &
655 & 64 &
665 & 48, 64, 116, 132, 204 &
673 & 84, 100, 138 &
679 & 22, 72 \\
689 & 112, 170 &
713 & 224 &
719 & 50, 58, 140, 154 &
721 & 90, 146 &
727 & 60, 170 \\
737 & 244 &
743 & 30, 48, 70, 168, 178 &
745 & 86, 112, 114 &
751 & 6, 188 &
761 & 28, 46 \\
767 & 56 &
769 & 40, 72 &
775 & 136, 186 &
785 & 198 &
791 & 10, 36, 180 \\
793 & 180 &
799 & 258 &
809 & 70, 192, 224 &
815 & 112 &
817 & 154, 210 \\
823 & 244 &
833 & 206, 228 &
839 & 18 &
841 & 48 &
847 & 92 \\
857 & 86, 90, 134, 212, 214, 246 &
865 & 76, 162, 288 &
871 & 126 &
881 & 26, 28 &
887 & 112, 190 \\
889 & 104, 240, 254 &
895 & 4 &
905 & 188 &
911 & 68, 126 &
913 & 158, 274 \\
919 & 12, 112, 130 &
937 & 240 &
943 & 8, 150 &
953 & 56 &
959 & 104, 188, 272 \\
961 & 6 &
967 & 12, 70 &
977 & 160 &
983 & 114 &
985 & 74 \\
991 & 266 &
1001 & 18, 118, 328 &
1007 & 32 &
1009 & 318 &
1015 & 62, 86 \\
1025 & 98, 102, 214 &
1031 & 220, 248 &
1033 & 36, 110, 166 &
1049 & 130, 252, 274 &
1055 & 8 \\
1057 & 66, 146, 242 &
1063 & 56 &
1079 & 94, 114 &
1081 & 8, 106, 116, 282 &
1087 & 80, 202, 210, 214 \\
1103 & 344, 346 &
1105 & 32, 222 &
1111 & 366 &
1121 & 306, 336, 338 &
1127 & 270 \\
1129 & 342 &
1135 & 12 &
1145 & 306 &
1151 & 30, 342 &
1153 & 204, 218, 246, 304 \\
1159 & 22 &
1169 & 38, 332 &
1177 & 42, 62, 96 &
1183 & 36 &
1193 & 340 \\
1199 & 38 &
1201 & 120 &
1207 & 232 &
1223 & 196 &
1225 & 78, 226 \\
1231 & 130, 238, 292, 322 &
1241 & 18 &
1247 & 30 &
1249 & 354 &
1255 & 222 \\
1265 & 184, 192, 266, 298, 382 &
1271 & 150, 406, 418 &
1273 & 56 &
1279 & 72 &
1289 & 68 \\
1297 & 66, 244, 320 &
1313 & 118, 308 &
1327 & 124, 316, 350 &
1337 & 34, 136 &
1343 & 116, 120 \\
1345 & 264 &
1351 & 50, 262, 370 &
1361 & 126 &
1369 & 120, 314 &
1375 & 42, 228, 246, 326 \\
1385 & 4, 242, 256 &
1391 & 28 &
1393 & 100, 114 &
1399 & 88, 180, 270, 380, 448 &
1415 & 94, 346 \\
1417 & 114 &
1423 & 76, 264, 378 &
1447 & 114, 242, 382 &
1463 & 296 &
1465 & 174 \\
1471 & 482 &
1481 & 390 &
1487 & 92 &
1489 & 84 &
1495 & 272, 418 \\
1505 & 82, 122, 398 &
1511 & 96 &
1513 & 230, 234 &
1519 & 164, 262 &
1529 & 20, 62, 214, 322 \\
1537 & 364, 460 &
1543 & 226, 366 &
1553 & 84, 306, 358 &
1559 & 324, 480 &
1561 & 186, 356 \\
1567 & 480 &
1577 & 270, 292, 412 &
1583 & 138, 280, 300 &
1591 & 152 &
1607 & 146, 176 \\
1615 & 504 &
1625 & 216, 408 &
1633 & 490 &
1639 & 310, 404 &
1649 & 220 \\
1655 & 534 &
1663 & 448 &
1673 & 30, 356 &
1679 & 14, 68, 400 &
1681 & 436 \\
1687 & 216, 432, 516, 552 &
1697 & 140, 256, 274 &
1703 & 514 &
1705 & 54, 248, 496 &
1721 & 100, 228 \\
1729 & 72 &
1735 & 294, 336 &
1745 & 234 &
1753 & 326 &
1769 & 294 \\
1775 & 162 &
1777 & 306, 520 &
1783 & 408, 448 &
1793 & 38 &
1799 & 104 \\
1807 & 146 &
1817 & 446 &
1823 & 228, 310 &
1831 & 440 &
1841 & 22 \\
1847 & 60, 256, 530 &
1849 & 600 &
1855 & 416 &
1865 & 246, 254, 448, 604 &
1873 & 370, 558 \\
1879 & 58, 284, 422 &
1889 & 464, 620 &
1903 & 310 &
1913 & 154 &
1919 & 240 \\
1921 & 156 &
1927 & 56, 592, 634 &
1937 & 310 &
1943 & 20, 176, 320, 644 &
1945 & 96, 98, 618 \\
1951 & 610 &
1961 & 332 &
1967 & 386, 506, 590 &
1969 & 182, 598 &
1975 & 378 \\
1985 & 224, 258, 558 &
1991 & 182 &
1993 & 254, 630 &
1999 & 244, 544 &
2009 & 18, 50 \\
2015 & 14, 186, 238 &
2017 & 110, 180, 476 &
2023 & 424 &
2033 & 384 &
2039 & 280, 364, 628 \\
2047 & 22, 410, 512 &
2057 & 604 &
2063 & 16, 190, 448, 570 &
2065 & 306, 656 &
2081 & 538 \\
2087 & 450, 506 &
2089 & 50, 282, 412, 580 &
2095 & 468, 496, 546 &
2111 & 440, 560 &
2113 & 184, 576 \\
2119 & 28, 144, 282 &
2129 & 178 &
2135 & 96, 696 &
2137 & 216, 602 &
2153 & 194, 496 \\
2159 & 2, 260 &
2167 & 280, 432 &
2183 & 54, 410 &
2185 & 108 &
2191 & 218, 240, 298 \\
2201 & 386, 532, 686 &
2207 & 146, 280 &
2209 & 272, 580 &
2225 & 726 &
2231 & 218, 246, 396 \\
2233 & 204, 590 &
2239 & 40, 408, 544 &
2249 & 42, 272, 360 &
2255 & 332, 496 &
2257 & 476 \\
2273 & 210, 358, 410 &
2279 & 92, 222, 490 &
2281 & 522 &
2287 & 486, 492 &
2303 & 34, 554 \\
2305 & 396, 552 &
2311 & 408, 700 &
2321 & 588 &
2327 & 44, 524, 634 &
2329 & 112, 430 \\
2335 & 594 &
2345 & 214, 772 &
2353 & 586 &
2359 & 272, 394 &
2369 & 98, 648, 764 \\
2375 & 66, 202, 724 &
2383 & 494, 778 &
2393 & 778 &
2399 & 690, 764 &
2401 & 388, 600, 736, 740, 780 \\
2407 & 450, 730, 772 &
2417 & 642 &
2423 & 34, 208, 370, 378 &
2425 & 226, 494 &
2431 & 210, 336, 696 \\
2441 & 270 &
2447 & 262, 420 &
2449 & 378 &
2455 & 714 &
2465 & 224 \\
2473 & 514 &
2479 & 24, 600 &
2495 & 232, 414, 666 &
2497 & 442 &
2513 & 804 \\
2519 & 414, 490, 558 &
2521 & 142 &
2537 & 450, 646, 812, 820 &
2543 & 80 &
2551 & 490, 816 \\
2561 & 620 &
2569 & 522 &
2585 & 228, 252, 774 &
2591 & 528 &
2593 & 454, 630, 644 \\
2599 & 612, 770 &
2609 & 164, 728 &
2615 & 318 &
2617 & 154, 270, 644 &
2623 & 318, 484, 554, 738 \\
2633 & 464 &
2639 & 48, 814 &
2641 & 402 &
2647 & 516, 530 &
2657 & 642 \\
2665 & 504 &
2671 & 418, 580 &
2681 & 218, 416, 682 &
2687 & 380, 602, 802 &
2689 & 12 \\
2695 & 684 &
2711 & 230 &
2713 & 280, 434 &
2719 & 34, 48, 140, 662, 778 &
2729 & 32, 280 \\
2737 & 156, 344, 404, 514 &
2743 & 760 &
2753 & 326 &
2759 & 430, 474 &
2761 & 36, 490, 510, 630 \\
2767 & 34, 206, 564 &
2783 & 56, 138, 734 &
2785 & 508, 812 &
2791 & 280 &
2807 & 90, 772 \\
2815 & 28, 252 &
2825 & 96, 404 &
2831 & 62, 582 &
2833 & 280, 550, 808 &
2839 & 602 \\
2849 & 38, 438 &
2855 & 278, 938 &
2857 & 252 &
2863 & 908 &
2879 & 734, 910 \\
2881 & 560 &
2887 & 554, 806 &
2897 & 546 &
2903 & 410 &
2905 & 246 \\
2911 & 308, 742 &
2921 & 160, 188, 332 &
2927 & 616 &
2945 & 504, 606, 792, 806, 868 &
2951 & 512 \\
2953 & 890 &
2959 & 228, 938 &
2969 & 192, 212 &
2975 & 64, 334, 542 &
2977 & 42, 146 \\
2983 & 110, 490, 566, 838 &
2993 & 680, 808 &
2999 & 350, 498, 554 & & &
\end{tabular}
\end{center}
\caption{Pairs $(n,s)$ with $7\leq n< 3000$ and $s$ even, such that $X^n+X^{n-s}+X^{n-2s}+X^{n-3s}+1\in\F_2[X]$ is irreducible.}\label{table2}
\end{table}
\begin{acknowledgements}
This work was initiated during the author's visit to the Federal University of Santa Catarina. 
The author is grateful to the anonymous reviewers for their valuable comments.
\end{acknowledgements}
%
%
%

%
\end{document}